\documentclass[a4paper,11pt]{amsart}
\usepackage[utf8]{inputenc}
\usepackage{amsmath}
\usepackage{amsfonts}
\usepackage{amsthm}
\usepackage{amssymb}
\usepackage{color}
\usepackage{tikz}
\usepackage{todonotes}

\newcommand{\set}[1]{\left\{ #1 \right\}}

\newcommand{\supp}[1]{\mathrm{Supp}(#1)}

\newcommand{\ex}[1]{\exp \left( #1 \right)}
\newcommand{\norm}[1] {\left\| #1 \right\|}
\newcommand{\h}{\mathcal{H}}
\newcommand{\dfn}{\mathrel{\mathop:}=}

\theoremstyle{plain}
\newtheorem{thm}{Theorem}[section]
\newtheorem{lem}[thm]{Lemma}

\newtheorem{prop}[thm]{Proposition}
\newtheorem{cor}[thm]{Corollary}

\newtheorem{con}[thm]{Convention}

\theoremstyle{definition}
\newtheorem{defn}[thm]{Definition}
\newtheorem{exmp}[thm]{Example}

\theoremstyle{remark}
\newtheorem{rem}[thm]{Remark}
\begin{document}
\title[Equivariant compression through subgroups]{Equivariant compression of certain direct limit groups and amalgamated free products}

\address{School of Mathematics,
University of Southampton, Highfield, Southampton, SO17 1BJ, United Kingdom.}
\author{Chris Cave}
\email[Chris Cave]{chriscave89@gmail.com}

\address{School of Mathematics,
University of Southampton, Highfield, Southampton, SO17 1BJ, United Kingdom.}
\author{Dennis Dreesen}
\email[Dennis Dreesen]{Dennis.Dreesen@soton.ac.uk}
\renewcommand{\thefootnote}{\fnsymbol{footnote}} 
\footnotetext{The first author is sponsored by the EPSRC, grant number EP/I016945/1. The second author is a Marie Curie Intra-European Fellow within the 7th European Community Framework Programme.}
\renewcommand{\thefootnote}{\arabic{footnote}} 
\begin{abstract}
We give a means of estimating the equivariant compression of a group $G$ in terms of properties of open subgroups $G_i\subset G$ whose direct limit is $G$. Quantifying a result by Gal, we also study the behaviour of the equivariant compression under amalgamated free products $G_1*_HG_2$ where $H$ is of finite index in both $G_1$ and $G_2$.
\end{abstract}

\maketitle
\section{Introduction}
The Haagerup property, which is a strong converse of Kazhdan's property (T), has translations
and applications in various fields of mathematics such as representation theory, harmonic
analysis, operator K-theory and so on. It implies the Baum--Connes conjecture and related Novikov conjecture \cite{CCJJV01}. We use the following definition of the Haagerup property.
\begin{defn}
A locally compact second countable group $G$ is said to satisfy the {\bf Haagerup property} if it admits a continuous proper affine isometric action $\alpha$ on some Hilbert space $\h$. Here, proper means that for every $M>0$, there exists a compact set $K\subset G$ such that $\|\alpha(g)(0)\|\geq M$ whenever $g\in G\setminus K$. We say that the action is continuous if the associated map $G\times \h \rightarrow \h, (g,v)\mapsto \alpha(g)(v)$ is jointly continuous.
\begin{con}
 Throughout this paper, all actions are assumed continuous and all groups will be second countable and locally compact.
\end{con}
\end{defn}
Recall that any affine isometric action $\alpha$ can be written as $\pi+b$ where $\pi$ is a unitary representation of $G$ and where $b:G \rightarrow \h$, $g\mapsto \alpha(g)(0)$ satisfies
\begin{equation}
\forall g,h \in G: \ b(gh)=\pi(g)b(h)+b(g).
\end{equation}
In other words, $b$ is a $1$-cocycle associated to $\pi$.

In \cite{GK02}, the authors define {\em compression} as a means to quantify {\em how strongly} a finitely generated group satisfies the Haagerup property. More generally, assume that $G$ is a compactly generated group. Denote by $S$ some compact generating subset and equip $G$ with the word length metric relative to $S$. Using the triangle inequality, one checks easily that any $1$-cocycle $b$ associated to a unitary action of $G$ on a Hilbert space is Lipschitz. On the other hand, one can look for the supremum of $r\in [0,1]$ such that there exists $C,D>0$ with
\[ \forall g\in G:\ \frac{1}{C} |g|^r-D \leq \| b(g) \| \leq C|g|+D. \]
\begin{defn}
The above supremum, denoted $R(b)$, is called the compression of $b$ and taking the supremum over all proper affine isometric actions of $G$ on all Hilbert spaces leads to the {\bf equivariant Hilbert space compression} of $G$, denoted $\alpha_2^\#(G)$. 
Suppose now that $G$ is no longer compactly generated but still has a proper length function. Then define $\alpha_2^{\#}(G)$ to be the supremum of $R(b)$ but over all \emph{large-scale Lipschitz} 1-cocycles.
\end{defn}

The equivariant Hilbert space compression contains information on the group. First of all, if $\alpha_2^\#(G)>0$, then $G$ is Haagerup. The converse was disproved by T. Austin in \cite{Austin2011}, where the author proves the existence of finitely generated amenable 
groups with equivariant compression $0$. Further, it was shown in \cite{GK02} that if for a finitely generated group $\alpha_2^\#(G)>1/2$, then $G$ is amenable. This result was generalized to compactly generated groups in \cite{dCTV07} and it provides some sort of converse 
for the well-known fact that amenability implies the Haagerup property. Much effort has been done to calculate the explicit equivariant compression value of several groups and classes of groups, see e.g. \cite{Stalder-Valette}, \cite{AGS06}, \cite{Gal2004}, 
\cite{T11}, \cite{ANP}.

Given two finitely generated group $G$ and $H$ the group $\bigoplus_H G$ is no longer finitely generated. However we can view $\bigoplus_H G$ as a subspace of $G \wr H$ and so equip $\bigoplus_H G$ with a natural proper metric. In this article we are motivated by comparing the compression of $\bigoplus_H G$ with $G \wr H$.  We assume that a given group $G$, equipped with a proper length function $l$, can be viewed as a direct limit of open (hence closed) subgroups $G_1\subset G_2 \subset G_3 \subset \ldots \subset G$. We equip each $G_i$ with the subspace metric from $G$. Our main 
objective will be to find bounds on $\alpha_2^\#(G)$ in terms of properties of the $G_i$. Note that, as each $G_i$ is a metric subspace of $G$, we have $\alpha_2^\#(G)\leq \inf_{i\in \mathbb{N}}\alpha_2^\#(G_i)$. The main challenge is to find  a sensible lower bound on $\alpha_2^\#(G)$. The key property that we introduce is the $(\alpha,l,q)$ polynomial property, which we shorten to $(\alpha,l,q)$-PP (see Definition \ref{intro:polp} below). 
 Precisely, we obtain the following result.
  \begin{thm}
  Let $G$ be a locally compact, second countable group equipped with a proper length function $l$. Suppose there exists a sequence of open subgroups $(G_i)_{i \in \mathbb{N}}$, each equipped with the restriction of $l$ to $G_i$, such that $\varinjlim G_i = G$ and $\alpha = \inf \{ \alpha_2^{\#}(G_i) \} > 0$. If $(G_i)_{i \in \mathbb{N}}$ has $(\alpha, l,q)$-PP, then there are the following two cases:
  \begin{equation*}
  l \geq q \Rightarrow \alpha_2^{\#}(G) \geq \frac{\alpha}{2l + 1}
 \end{equation*}
or,
\begin{equation*}
 l \leq q \Rightarrow \alpha_2^{\#}(G) \geq \frac{\alpha}{l + q +1}.
\end{equation*}
\label{thm:main}
\label{lower bound for compression of direct limits of groups}
 \end{thm}
 
 We use this result to obtain a lower bound of the compression of the following examples. Let $F \colon [0,1] \times \mathbb{R}^{\geq 0} \to \mathbb{R}$ be the function
\[
 F(\alpha,d) = \begin{cases}
                d(2\alpha -1) &  \mbox{if $2\alpha \geq 1$} \\
                0 & \mbox{otherwise.}
               \end{cases}
\]
\begin{thm}
 Let $G$ and $H$ be finitely generated groups where $H$ has polynomial growth of degree $d \geq 1$. Then
 \[
  \alpha^{\#}_2\left(\bigoplus_H G\right) \geq \frac{\alpha_2^{\#}(G)}{1 + F(\alpha^{\#}_2(G),d)+ 2\alpha^{\#}_2(G)(1+d)}
 \]
 where $\bigoplus_H G$ is equipped with the subspace metric from $G \wr H$.

\end{thm}
Our result also allows to consider spaces $\bigoplus_H G_h$ where $G_h$ actually depends on the parameter $h\in H$. For example, we take a collection of finite groups $F_i$ with $F_0=\{0\}$ and look at $G=\bigoplus_{i\in \mathbb{N}} F_i$. This is the first available lower bound for the equivariant compression of groups of this type.

\begin{thm}
 Let $\set{F_i}_{i \in \mathbb{N}}$ be a collection of finite groups. Equip $G = \bigoplus_{i \in \mathbb{N}} F_i$ with the length function $l(g) = \min \set{n \in \mathbb{N} : g \in \oplus_{i = 0}^n F_i}$. Then $\alpha_2^{\#}(G,l) > 1/3$.
\end{thm}

We give a proof of Theorem \ref{thm:main} in Section \ref{theproof} and apply to these concrete examples in Section \ref{example}. Note that our result can also be viewed as a study of the behaviour of equivariant compression under direct limits. 
The behaviour of the Haagerup property and the equivariant compression under group constructions  has been studied extensively (see e.g. \cite{NP08}, \cite{Dreesen}, Chapter 6 of \cite{CCJJV01}, \cite{CSV12}, \cite{AntolinDreesen}).

In Section \ref{gal}, we quantify part of \cite{Gal2004} to study the behaviour of the equivariant compression under certain amalgamated free products $G_1*_H G_2$ where $H$ is of finite index in both $G_1$ and $G_2$. Suppose $H$ is a closed finite index subgroup inside groups compactly generated groups $G_1$ and $G_2$ and there exists proper affine isometric actions $\beta_i \colon G_i \to \mathrm{Aff}(V_i)$ on Hilbert spaces $V_i$. In \cite{Gal2004}, the author shows that if there exists a non-trivial closed subspace $W \subset V_1 \cap V_2$ that is fixed by the restricted actions $\beta_i |_H$ then the product $G_1 \ast_H G_2$ also admits a proper affine isometric action on a Hilbert space. We quantify this result.
\begin{thm}
 With the above assumptions $\alpha_2^{\#}(G_1 \ast_H G_2) \geq \frac{\alpha_2^{\#}(H)}{2}$
\end{thm}

\section{The equivariant compression of direct limits of groups \label{limit}}
\subsection{\label{pre}Preliminaries and formulation of the main result}
Suppose $G$ is a locally compact second countable group equipped with a proper length function $l$, i.e. closed $l$-balls are compact. Assume that there exists a sequence of open subgroups $G_i\subset G$ such that $\varinjlim G_i = G$, i.e. $G$ is the direct limit of the $G_i$. We equip each $G_i$ with the restriction of $l$ to $G_i$. It will be our goal to find bounds on $\alpha_2^\#(G)$ in terms of the $\alpha_2^\#(G_i)$. 
Clearly, as the $G_i$ are subgroups then an upper bound of the equivariant compression is the infimum of the equivariant compressions of the $G_i$. The challenge is to find a sensible lower bound. The next example will show that it is not enough to only consider the $\alpha_2^\#(G_i)$.

\begin{exmp}
 Consider the wreath product $\mathbb{Z}\wr \mathbb{Z}$ equipped with the standard word metric relative to $\{(\delta_1,0),(0,1)\}$, where $\delta_1$ is the characteristic function of $\{0\}$. Let $\mathbb{Z}^{(\mathbb{Z})} = \set{f \colon \mathbb{Z} \to \mathbb{Z} : \mbox{ $f$ is has finite support}}$ be equipped with the subspace metric from $\mathbb{Z} \wr \mathbb{Z}$. Consider the direct limit of groups
\begin{equation*}
 \mathbb{Z} \hookrightarrow \mathbb{Z}^3 \hookrightarrow \mathbb{Z}^{5} \cdots \hookrightarrow \mathbb{Z}^{(\mathbb{Z})}
\end{equation*}
where $\mathbb{Z}^{2n+1}$ has the subspace metric from $\mathbb{Z}^{(\mathbb{Z})}$. This metric is quasi-isometric to the standard word metric on $\mathbb{Z}^{2n+1}$ and so each term has equivariant compression 1. So $\mathbb{Z}^{(\mathbb{Z})}$ is a direct limit of groups with equivariant compression 1 but by \cite{AGS06} has equivariant compression less than 3/4. On the other hand the sequence
\begin{equation*}
 \mathbb{Z} \to \mathbb{Z} \to \cdots \to \mathbb{Z}
\end{equation*}
is a sequence of groups with equivariant compression 1 and the equivariant compression of the direct limit is 1.
$\qed$
\end{exmp}

Given a sequence of $1$-cocycles $b_i$ of $G_i$, then in order to predict the equivariant compression of the direct limit, it will  be necessary to incorporate more information on the growth behaviour of the $b_i$ than merely the compression exponent $R(b_i)$. The growth behaviour of $1$-cocycles can be completely caught by so called {\em conditionally negative definite functions} on the group (See Proposition \ref{prop:nsiscondneg} and Theorem \ref{thm:condnegisns} below).

\begin{defn}
A continuous map $\psi:G\rightarrow \mathbb{R}^+$ is called {\em conditionally negative definite} if $\psi(g)=\psi(g^{-1})$ for every $g\in G$ and if for all $n\in \mathbb{N}, \ \forall g_1,g_2,\ldots ,g_n \in G$ and all $a_1,a_2,\ldots, a_n\in \mathbb{R}$ with $\sum_{i=1}^n a_i=0$,  we have
\[ \sum_{i,j} a_i a_j \psi(g_i^{-1}g_j)\leq 0 .\]
\end{defn}
\begin{prop}[Example 13, page 62 of \cite{Harpe1989}]
 Let $\mathcal{H}$ be a Hilbert space and $b \colon G \to \mathcal{H}$ a 1-cocycle associated to a unitary representation. Then the map $\psi \colon G \to \mathbb{R}$, $g \mapsto \norm{b(g)}^2$ is a conditionally negative definite function on $G$. 
\label{prop:nsiscondneg}
\end{prop}

\begin{thm}[Proposition 14, page 63 of \cite{Harpe1989}]
 Let $\psi \colon G \to \mathbb{R}$ be a conditionally negative definite function on a group $G$. Then there exists an affine isometric action $\alpha$ on a Hilbert space $\mathcal{H}$ such that the associated 1-cocycle satisfies $\psi(g) = \norm{b(g)}^2$. \label{thm:condnegisns}
\end{thm}

These two results imply that we can pass between conditionally negative definite functions and 1-cocycles associated to unitary actions. 

\begin{defn} \label{intro:polp}
  Let $G$ be a group equipped with a proper length function $l$ and suppose that $(G_i)_{i \in \mathbb{N}}$ is a normalized nested sequence of open subgroups such that $\varinjlim G_i = G$. Assume that $\alpha:= \inf_{i\in \mathbb{N}}\alpha_2^\#(G_i) \in (0,1]$ and $l,q \geq 0$.  The sequence $(G_i)_i$ has the \emph{$(\alpha, l,q)$-polynomial property} ($(\alpha, l,q)$-PP) if there exists:
  \begin{enumerate}
   \item a sequence $(\eta_i)_i\subset \mathbb{R}^+$ converging to $0$ such that $\eta_i<\alpha$ for each $i\in \mathbb{N}$,
   \item $(A_i, B_i)_{i \in \mathbb{N}} \subset \mathbb{R}^{>0} \times \mathbb{R}^{\geq 0}$,
   \item a sequence of $1$-cocycles $(b_i \colon G_i \to \mathcal{H}_i)_{i \in \mathbb{N}}$, where each $b_i$ is associated to a unitary action $\pi_i$ of $G_i$ on a Hilbert space $\mathcal{H}_i$
   \end{enumerate}
such that 
  \begin{equation*}
   \frac{1}{A_i}|g|^{2\alpha - \eta_i} - B_i \leq \|b_i(g)\|^2 \leq A_i |g|^2 + B_i \quad \forall g \in G_i, \forall i \in \mathbb{N}
  \end{equation*}
  and there is $C,D>0$ such that $A_i\leq Ci^l$, $B_i \leq Di^q$ for all $i \in \mathbb{N}$.
 \end{defn}
Note that the only real restrictions are the inequalities  $A_i\leq Ci^l$, $B_i \leq Di^q$: we exclude sequences $A_i,B_i$ that grow faster than any polynomial.  The intuition is that equivariant compression is a polynomial property (this follows immediately from its definition), so that sequences $A_i,B_i$ growing faster than any polynomial would be too dominant and one would lose all hope of obtaining a lower bound on $\alpha_2^\#(G)$. On the other hand, if the $A_i$ and $B_i$ grow polynomially, then one can use compression to somehow compensate for this growth. One then obtains a strictly positive lower bound on $\alpha_2^\#(G)$ which may decrease depending on how big $l$ and $q$ are.
We have the following useful characterisation of $(\alpha,l,q)$-polynomial property.
\begin{lem}
 Let $G$ be a locally compact second countable group and $l$ is a proper length metric. Suppose there exists a sequence of open subgroups $(G_i)_{i \in \mathbb{N}}$ such that $\varinjlim G_i = G$. If each $G_i$ are equipped with the restricted length metric from $G$ then $(G_i)_{i \in \mathbb{N}}$ has the $(\alpha,l,q)$-polynomical property if and only if there exists $C,D > 0$ such that for all $\varepsilon > 0$ there exists
 \begin{enumerate}
  \item a sequence $(A_i,B_i)_{i \in \mathbb{N}} \subset \mathbb{R}^{>0} \times \mathbb{R}^{\geq 0}$ such that $A_i \leq Ci^l$ and $B_i \leq Di^q$;
  \item a sequence of 1-cocycles $(b_i \colon G_i \to \h_i)_{i \in \mathbb{N}}$
 \end{enumerate}
such that
\[
 \frac{1}{A_i}|g|^{2 \alpha - \varepsilon} - B_i \leq \norm{b_i(g)}^2 \leq A_i|g|^2 +B_i \quad \forall g \in G_i, \forall i \in \mathbb{N}
\]
\end{lem}

\begin{proof} The ``if'' direction is obvious. For the ``only if'' direction fix $\varepsilon > 0$ and suppose $(G_i)_{i \in \mathbb{N}}$ has the $(\alpha,l,q)$-polynomial property with respect to  sequences $(\eta_i)_{i \in \mathbb{N}}$ and $(b_i \colon G_i \to \mathcal{H}_i)_{i \in \mathbb{N}}$. Choose $N \in \mathbb{N}$ large enough so that $\eta_k < \varepsilon$ for all $k \geq N$. Thus $b_k \colon G_k \to \mathcal{H}_k$ satisfies the above conditions for all $k \geq N$. For $k \leq N$ we take the restriction of $b_N$ to $G_k$ to obtain the sequence satisfying the above conditions for all $k \in \mathbb{N}$.
\end{proof}

\begin{prop}
 Let $G$ be a locally compact second countable group and suppose there exists a sequence of open subgroups $(G_i)_{i \in \mathbb{N}}$ such that $\varinjlim G_i = G$. If $\alpha \dfn \alpha^{\#}_2(G) > 0$ then $(G_i)_{i \in \mathbb{N}}$ has $(\alpha,0,0)$-polynomical property.
\end{prop}

\begin{proof}
 For all $0 < \varepsilon < \alpha$ there exists a 1-cocycle $b$ such that
 \[
  \frac{1}{A}|g|^{\alpha - \varepsilon} - B \leq \norm{b(g)} \quad \forall g \in G
 \]
The restriction of $b$ to each $G_i$ is a 1-cocycle and gives $(G_i)_{i \in \mathbb{N}}$ the $(\alpha,0,0)$-polynomial property.
\end{proof}
Combining this with Theorem \ref{thm:main} we have the following consequence which confirms our intuition.
\begin{cor}
 Let $G$ be a locally compact second countable group with a proper length function $l$. If there exists a sequence of open subgroups $(G_i)_{i \in \mathbb{N}}$ such that $\varinjlim G_i = G$ then $(G_i)_{i \in \mathbb{N}}$ has the $(\alpha,l,q)$-polynomial property for some $\alpha \in (0,1]$ and $l,q \geq 0$ if and only if $\alpha_{2}^{\#}(G) > 0$
\end{cor}

\subsection{\label{theproof} The proof of Theorem \ref{thm:main} }

 \begin{proof}[Proof of Theorem \ref{thm:main}]
 First we can assume that $l$ is uniformly discrete. That is there exists a $c > 0$ such that $l(x) > c$ for all $x \in G \setminus \set{e}$. This is because given a length function $l$ one can define a new length function $l'$ such that $l'(x) = 1$ whenever $0< l(x) \leq 1$ and $l'(x) = l(x)$ when $l(x) \geq 1$. Hence $l'$ will be quasi-isometric to $l$ and so will not change the compression of $G$ or $G_i$.
 
  Take sequences $(\psi_i \colon G_i \to \mathbb{R})_{i \in \mathbb{N}}$, $(\eta_i)_i$ and $(A,B) = (A_i, B_i)_{i \in \mathbb{N}} \subset \mathbb{R}^{>0} \times \mathbb{R}^{\geq 0}$ satisfying the conditions of $(\alpha,l,q)$-PP (see Definition \ref{intro:polp}). We assume here, without loss of generality, that the sequences $(A_i)_i,(B_i)_i$ are non-decreasing.
	
  For each $G_i$, define a sequence of maps $(\varphi^i_k \colon G_i \to \mathbb{R})_{k \in \mathbb{N}}$ by
  \begin{equation*}
   \varphi^i_k (g) = \begin{cases} \exp \left( \frac{-\psi_i (g)}{k}\right) & \mbox{if $g \in G_i$}\\
   0 & \mbox{otherwise}.                      
                     \end{cases}
  \end{equation*}
	Note that each $\varphi_k^i$ is continuous as $G_i$ is open and also closed, being the complement of $\cup_{g\notin G_i} gG_i$.
  By $(\alpha,l,q)$-PP, for all $i,k \in \mathbb{N}$, we have 
  \begin{multline*}
   \ex{\frac{-A_i |g|^2 - B_i}{k}} \leq \varphi_k^i(g)  \quad \forall g \in G_i, \mbox{ and } \\ \varphi_k^i(g) \leq \ex{\frac{-|g|^{2\alpha- \eta_i} + A_iB_i}{A_ik}} \quad \forall g \in G. 
  \end{multline*}
Fix some $p>0$, set $J(i) = (A_{i}+B_i)i^{1+p}$ and define $\overline{\psi} \colon G \to \mathbb{R}$ by
  \begin{equation*}
   \overline{\psi}(g) = \sum_{i \in \mathbb{N}} 1 - \Phi_i(g),
  \end{equation*}
	 where $\Phi_i(g):=\varphi_{J(i)}^{i} (g)$.
 To check that $\overline{\psi}$ is well defined, choose any $g\in G$ and note that for $i>\lvert g \rvert$, we have $g\in G_{i}$ and so $\varphi_k^{i}(g)\geq \exp(\frac{-A_{i} |g|^2 - B_{i}}{k})$. Hence
 \begin{align*}
 \sum_{i > |g|} 1 - \Phi_i(g) &\leq \sum_{i>\lvert g \rvert} 1 -\exp\left(\frac{-A_i\lvert g \rvert^2-B_i}{(A_i+B_i)i^{1+p}}\right) \\
 &\leq \sum_{i>\lvert g \rvert} 1-\exp \left(\frac{-\lvert g \rvert^2}{i^{1+p}}\right) \\
 & \leq \sum_{i>\lvert g \rvert} \frac{\lvert g \rvert^2}{i^{1+p}}=\lvert g \rvert^2 \sum_{i>\lvert g \rvert} \frac{1}{i^{1+p}}
 \end{align*}

As \[  \overline{\psi}(g) = \sum_{i=1}^{|g|} 1 - \Phi_i(g) + \sum_{i > |g|} 1 - \Phi_i(g), \]
we see that $\overline{\psi}$ is well defined and that it can be written as a limit of continuous functions converging uniformly over compact sets. Consequently, it is itself continuous. By Schoenberg's theorem (see \cite[Theorem 5.16]{Harpe1989}), all of the maps $\varphi_k^i$ are positive definite on $G_i$ and hence on $G$ (see \cite[Section 32.43(a)]{Hewitt}). In other words,
\[ \forall n\in \mathbb{N}, \ \forall a_1,a_2,\ldots , a_n\in \mathbb{R}, \forall g_1,g_2, \ldots ,g_n\in G: \ \sum_{i,j=1}^n a_i a_j \varphi_k^i(g_i^{-1}g_j)\geq 0.\]
Hence, $\overline{\psi}$ is a conditionally negative definite map. Moreover, using that $l$ is uniformly discrete, we can find a constant $E>0$ such that
\begin{equation}
\overline{\psi}(g) \leq \lvert g \rvert + \lvert g \rvert^2 \sum_{i>\lvert g \rvert} \frac{1}{i^{1+p}} \leq E\lvert g \rvert^2 \label{welldefined}
\end{equation}
so the $1$-cocycle associated to $\overline{\psi}$ via Theorem \ref{thm:condnegisns} is large-scale Lipschitz.

Let us now try to find the compression of this 1-cocycle. Set $VI \colon \mathbb{N} \to \mathbb{R}$ to be the function
\begin{equation*}
 VI(i) = (A_{i} J(i) \ln (2) + A_{i} B_{i})^{\frac{1}{2\alpha - \eta_{i}}}
\end{equation*}
One checks easily that
\begin{equation} 
|g| \geq VI(i) \Rightarrow \Phi_i(g)=\varphi_{J(i)}^{i} (g) \leq \frac{1}{2}.
\label{eq:VI}
\end{equation}
To make the function $VI$ more concrete, let us look at the values of $A_{i}, B_{i}$ and $J(i)$. Recall that by assumption, we have $A_i\leq Ci^l, B_i\leq Di^q$. Hence for $i$ sufficiently large, we have $J(i)\leq (Ci^{l}+D i^{q}) i^{1+p} \leq F i^X$ where $F$ is some constant and $X=1+p+\max(l,q)$. We thus obtain that there is a constant $K>0$ such that for every $i$ sufficiently large (say $i>I$ for some $I\in \mathbb{N}_0$),
\begin{equation*}
 VI(i)\leq K i^{Y/(2\alpha-\eta_i)},
\end{equation*} 
where
\begin{align*}
Y&=\max(X+l,l+q)\\
&= \max(1+p+2l,1+p+l+q)
\end{align*}
As the sequence $\eta_i$ converges to $0$, we can choose any $\delta>0$ and take $I>0$ such that in addition $\eta_i<\delta$ for $i>I$. We then have for all $i>I$ that
\[ VI(i)\leq K i^{Y/(2\alpha-\delta)}.\]
Together with Equation \eqref{eq:VI}, this implies that for $i>I$,
\begin{equation} 
|g| \geq Ki^{Y/(2\alpha-\delta)} \Rightarrow \Phi_i(g)=\varphi_{J(i)}^{i} (g) \leq \frac{1}{2}.
\end{equation}
For every $g\in G$, set
\begin{equation*}
c(g)_{p,\delta} = \sup \set{i \in \mathbb{N} \ | \ Ki^{Y/(2\alpha-\delta)}\leq |g| }. 
\end{equation*}
We then have for every $g\in G$ with $|g|$ large enough, that

\begin{align*}
 \overline{\psi}(g) &\geq  \sum_{i = 1}^{c(g)_{p,\delta}} 1 - \varphi_{J(i)}^{i}(g) \\
&\geq  \sum_{i = I+1}^{c(g)_{p,\delta}} 1/2 = \frac{c(g)_{p,\delta}-I}{2}
\end{align*}
As $c(g)_{p,\delta}\geq (\frac{\lvert g \rvert}{K})^{(2\alpha-\delta)/Y}-1$, we conclude that $R(b)\geq \frac{2\alpha-\delta}{2\max(1+p+2l,1+p+l+q)}$. As this is true for any small $p,\delta>0$, we can take the limit for $p,\delta \to 0$ to obtain $\alpha_2^{\#}(G)\geq \frac{\alpha}{\max(1+2l,1+l+q)}$. Hence, we have the following two cases:
 \begin{equation*}
  l \geq q \Rightarrow \alpha_2^{\#}(G) \geq \frac{\alpha}{1+2l}
 \end{equation*}
or,
\begin{equation*}
 l \leq q  \Rightarrow \alpha_2^{\#}(G) \geq \frac{\alpha}{l + q +1}. \qedhere
\end{equation*}
 \end{proof}

\subsection{\label{example} Examples}
Let $F \colon [0,1] \times \mathbb{R}^{\geq 0} \to \mathbb{R}$ be the function
\[
 F(\alpha,d) = \begin{cases}
                d(2\alpha -1) &  \mbox{if $2\alpha \geq 1$} \\
                0 & \mbox{otherwise.}
               \end{cases}
\]
\begin{thm}\label{thm:polynomial growth}
 Let $G$ and $H$ be finitely generated groups where $H$ has polynomial growth of degree $d \geq 1$. Then
 \[
  \alpha^{\#}_2\left(\bigoplus_H G\right) \geq \frac{\alpha_2^{\#}(G)}{1 + F(\alpha^{\#}_2(G),d)+ 2\alpha^{\#}_2(G)(1+d)}
 \]
 where $\bigoplus_H G$ is equipped with the subspace metric from $G \wr H$.

\end{thm}
\begin{rem}
 Theorem 1.3.~from \cite{Li2010} provides a lower bound to the compression of $G \wr H$. Under the assumptions in Theorem \ref{thm:polynomial growth}, Theorem 1.3.~in \cite{Li2010} gives a lower bound $\alpha_2^{\#}(G \wr H) \geq \alpha_1^{\#}(G)/2$. As this bound is in terms of $L^1$-compression, this makes comparison between between the bound in Theorem \ref{thm:polynomial growth} and \cite[Theorem 1.3.]{Li2010} difficult. However it is known that $\alpha^{\#}_2(G) \leq \alpha^{\#}_1(G) \leq 2 \alpha_2^{\#}(G)$ for all finitely generated groups $G$, see the proof of Theorem 1.1.~and Theorem 1.3.~in \cite{Li2010} and \cite[Lemma 2.3.]{NP08}.
 
 We use this to show that under some circumstances the above lower bound is larger than the bound provided in \cite[Theorem 1.3.]{Li2010}. Suppose that $\alpha_1^{\#}(G)/2 < \alpha^{\#}_2(G)$. Then there exists a $c > 0$ such that $\frac{2\alpha^{\#}_2(G)}{\alpha_1^{\#}(G)} > 1 +c$. If $\alpha^{\#}_2(G) \leq \min \set{\frac{c}{2(1+d)}, 1/2}$ then by Theorem \ref{thm:polynomial growth}
 \[
  \alpha_2^{\#}(\oplus_H G) \geq \frac{\alpha_2^{\#}(G)}{1+c} > \frac{\alpha_1^{\#}(G)}{2}
 \]
 Unfortunately the values of $\alpha_2^{\#}$ are not so well understood and at the time of writing the only know values for $\alpha_2^{\#}$ are 1, 1/2, 0 and $\frac{1}{2 - 2^{1-k}}$ for $k \in \mathbb{N}$ \cite{AGS06, NP08,Austin2011}. In the non-equivariant case any value for compression can be achieved \cite{ADS09}. It is likely that there exists groups such that $\alpha_2^{\#}$ takes values strictly between 0 and 1/2 in which case our theorem can be applied to provide larger lower bounds than $\alpha_1^{\#}(G) / 2$.
\end{rem}

\begin{proof}
We consider $\bigoplus_H G$ to be the group of functions $\mathbf{f} \colon H \to G$ that have finite support. Let $\mathbf{f} \in \bigoplus_H G$ and let $\supp{\mathbf{f}} = \set{h_1, \ldots, h_n} \subset H$. Set the length of $\mathbf{f}$ as follows
\begin{multline*}
 |\mathbf{f}|_{G \wr H} = \inf_{\sigma \in S_n} \left( d_{H}(1,h_{\sigma(1)}) + \sum_{i = 1}^{n} d_H(h_{\sigma(i)}, h_{\sigma(i+1)}) + d_{H}(h_{\sigma(n)}, 1) \right)\\ + \sum_{h \in H} |\mathbf{f}(h)|_G.
\end{multline*}
This is the induced length metric from $G \wr H$ and so this is a proper length function on $\bigoplus_H G$. Consider the following group
\[
G_i = \set{\mathbf{f} \colon H \to G : \supp{\mathbf{f}} \subset B(1,i)} 
\]
and set $n_i = |B(1,i)|$. Each $G_i$ is finitely generated and the restricted wreath metric to $G_i$ is proper and left invariant so the wreath metric and the word metric are quasi-isometric. In particular 
\[
 |\mathbf{f}|_{G \wr H} - 2i|B(1,i)| \leq \sum_{h \in B(1,i)} |\mathbf{f}(h)|_{G} \leq |\mathbf{f}|_{G \wr H}
\]
for all $\mathbf{f} \in G_i$. By \cite[Proposition 4.1. and Corollary 2.13.]{GK04} it follows that $\alpha_2^{\#}(G_i) = \alpha_2^{\#}(G)$ for all $i \in \mathbb{N}$. 
%
%
Set $0 < \alpha < \alpha_2^{\#}(G)$ and consider a 1-cocyle $b \colon G \to \h$ such that
\[
 \frac{1}{C} |g|_G^{2\alpha} \leq \norm{b(g)}^2 \leq C |g|_G^2.
\]
Enumerate $B(1,i)$ so that $\set{h_1, \ldots, h_{n_i}} = B(1,i)$ and define a 1-cocycle $b_{i} \colon G_i \to \h^{n_i}$, where $b_i(\mathbf{f}) = (b(\mathbf{f}(h_1)),\ldots, b(\mathbf{f}(h_{n_i})))$. If $|\mathbf{f}|_{G \wr H} > 4i|B(1,i)|$, then
\begin{multline*}
 \norm{b_i(\mathbf{f})}_{1/\alpha} = \left( \sum_{j=1}^{i}\norm{b(\mathbf{f}(h_{n_j}))}^{1/\alpha} \right)^{\alpha} \geq \frac{1}{C^{1/\alpha}} \left(\sum_{j=1}^{i}|\mathbf{f}(h_{n_j})|_G \right)^{\alpha} \\  \geq \frac{1}{C^{1/\alpha}}\left(|\mathbf{f}|_{G \wr H} - 2i|B(1,i)|\right)^{\alpha} 
 \geq \frac{1}{2C^{1/\alpha}}|\mathbf{f}|_{G \wr H}^{\alpha}.
\end{multline*}
%
%
If $2 \alpha < 1$ then $\norm{b_i(\mathbf{f})}_2 \geq \norm{b_i(\mathbf{f})}_{1 / \alpha}$ for all $\mathbf{f} \in G_i$ and so it follows that 
\[
\frac{1}{4C^{2/\alpha}} |\mathbf{f}|_{G \wr H}^{2\alpha} - \frac{ i^{2\alpha}}{C}|B(1,i)|^{2\alpha} \leq \norm{b_i(\mathbf{f})}_2^2
\]
for all $\mathbf{f} \in G_i$. Hence $(G_i)_{i \in \mathbb{N}}$ has the $(\alpha,0,2\alpha(1+d))$ polynomial property.

If $2 \alpha \geq 1$ then by H\"{o}lder's inequality $\norm{b_i(\mathbf{f})}_2 \geq n_i^{\frac{1-2\alpha}{2}} \norm{b_i(\mathbf{f})}_{1 / \alpha}$ for all $\mathbf{f} \in G_i$ and so it follows that
\[
\frac{1}{4C^{2/\alpha}|B(1,i)|^{2\alpha - 1}} |\mathbf{f}|_{G \wr H}^{2\alpha} - \frac{ i^{2\alpha}}{C}|B(1,i)|^{2\alpha} \leq \norm{b_i(\mathbf{f})}^2_2.
\]
for all $\mathbf{f} \in G_i$. Hence $(G_i)_{i \in \mathbb{N}}$ has the $(\alpha, d(2\alpha - 1),2\alpha(1+d))$ polynomial property. Thus by Theorem \ref{thm:main} and that $\alpha, d \geq 0$ it follows that
\[
 \alpha^{\#}_2(\bigoplus_H G) \geq \frac{\alpha}{1 + F(\alpha,d) + 2 \alpha(1+d)}
\]
for all $\alpha < \alpha^{\#}_2(G)$ and so the statement of the theorem holds.
\end{proof}

\begin{thm}
 Let $\set{F_i}_{i \in \mathbb{N}}$ be a collection of finite groups such that $F_0 = \set{1}$. Let $G = \bigoplus_{i \in \mathbb{N}} F_i$ be equip with the proper length function $l(g) = \min \set{n \in \mathbb{N} : g \in \oplus_{i=0}^n F_i}$. Then $\alpha^{\#}_2(G) \geq 1/3$.
\end{thm}

\begin{proof}
 Set $G_i = \bigoplus_{j=0}^i F_j$ and observe that $\alpha_2^{\#}(G_i) = 1$ as $G_i$ is finite for all $i \in \mathbb{N}$. Define $f_i \colon G_i \to \mathbb{R}$ to be the 0-map. This is clearly a 1-cocycle and satisfies
 \[  \forall g \in G_i: \ l(g)^2-i^2 \leq |f_i(g)|^2 \leq l(g)^2+i^2.\]
Hence $(G_i)_{i \in \mathbb{N}}$ has the (1,0,2)-polynomial property. Thus $\alpha_2^{\#}(G) \geq 1/3$.
 \end{proof}

%


\begin{exmp} \label{necessity}
 We will use \cite{ADS09} to provide an example of a sequence that does not have $(\alpha,l,q)$-polynomial property for any $\alpha \in (0,1]$ and $l,q > 0$. Let $\Pi_k$, $k \geq 1$ be a sequence of Lafforgue expanders that do not embed into any uniformly convex Banach space \cite{Lafforgue2008}. These are finite factor groups $M_k$ of a lattice $\Gamma$ of $\mathrm{SL}_3(F)$ for a local field $F$.

For every $\alpha \in [0,1]$ there exists a finitely generated group $G$ and a sequence of scaling constants $\lambda_k$ such that $\lambda_k \Pi_k$ has compression $\alpha$ and $G$ is quasi-isometric to $\lambda_k \Pi_k$. Furthermore $G$ contains the free product $\ast_k M_k$ as a subgroup. Let $\alpha = 0$ and let $G$ and the scaling constants $\lambda_k$ be such that $G$ has compression 0. We can equip $\ast_k M_k$ with a proper left invariant metric coming from $G$. Hence we have a sequence
\[
 M_1 \hookrightarrow M_1 \ast M_2 \hookrightarrow \cdots \hookrightarrow \ast_{k=1}^n M_k \hookrightarrow \cdots \hookrightarrow \ast_k M_k
\]
For each $n > 0$, $\ast_{k=1}^n M_k$ has equivariant compression 1/2 \cite[Theorem 1.4.]{Dreesen} however the limit group $\ast_k M_k$ contains a quasi-isometric copy of $\lambda_k \Pi_k$ and so has compression 0. Thus this sequence can not have the $(\alpha,l,q)$-polynomial property for any $\alpha \in (0,1]$ and $l,q > 0$.

\end{exmp}

 \section{The behaviour of compression under free products amalgamated over finite index subgroups \label{gal}}
It is known that the Haagerup property is not preserved under amalgamated free products. Indeed, $(SL_2(\mathbb{Z})\rtimes \mathbb{Z}^2, \mathbb{Z}^2)$ has the relative property $(T)$. So $SL_2(\mathbb{Z})\rtimes \mathbb{Z}^2=(\mathbb{Z}_6\rtimes \mathbb{Z}^2)*_{(\mathbb{Z}_2 \rtimes \mathbb{Z}^2)} (\mathbb{Z}_4\rtimes \mathbb{Z}^2)$ is not Haagerup. In \cite{Gal2004}, S.R. Gal proves the following result.
\begin{thm}
 Let $G_1$ and $G_2$ be finitely generated groups with the Haagerup property that have a common finite index subgroup $H$. For each $i=1,2$, let $\beta_i$ be a proper affine isometric action of $G_i$ on a Hilbert space $V_i(=l^2(\mathbb{Z}))$. Assume that $W<V_1\cap V_2$ is invariant under the actions $(\beta_i)_{\mid H}$ and moreover that both these (restricted) actions coincide on $W$. Then $G_1 \ast_H G_2$ is Haagerup.
\end{thm}
Under the same conditions as above, we want to give estimates on $\alpha_2^\#(G_1*_HG_2)$ in terms of the equivariant Hilbert space compressions of $G_1,G_2$ (see Theorem \ref{thm:amalgamated} below). Note that the following lemma shows that $\alpha_2^\#(G_1)=\alpha_2^\#(H)=\alpha_2^\#(G_2)$ when $H$ is of finite index in both $G_1$ and $G_2$. We are indebted to Alain Valette for this lemma and its proof. The notation $\alpha_p^\#$ refers to the equivariant {\bf $L_p$-compression} for some $p\geq 1$. It is defined in exactly the same way as $\alpha_2^\#$ except that one considers affine isometric actions on $L_p$-spaces instead of $L_2$-spaces.

\begin{lem} Let $G$ be a compactly generated, locally compact group, and let $H$ be an open, finite-index subgroup of $G$. Then $\alpha_p^\#(H)=\alpha_p^\#(G)$.
\end{lem}

\begin{proof} As $H$ is embedded $H$-equivariantly, quasi-isometrically in $G$, we have $\alpha_p^\#(H)\geq \alpha_p^\#(G)$. To prove the converse inequality, we may assume that $\alpha_p^\#(H)>0$. Let $S$ be a compact generating subset of $H$. Let $A(h)v=\pi(h)v + b(h)$ be an affine isometric action of $H$ on $L^p$, such that for some $\alpha<\alpha_p^\#(H)$ we have $\|b(h)\|_p \geq C|h|_S^\alpha$, for every $h\in H$. Now we induce up the action $A$ from $H$ to $G$, as on p.91 of \cite{BHV}\footnote{We seize this opportunity to correct a misprint in the definition of the vector $\xi_0$ in that construction in p.91 of \cite{BHV}.}. The affine space of the induced action is
\begin{equation*}
E \mathrel{\mathop:}= \{f:G\rightarrow L^p: f(gh)=A(h)^{-1}f(g),\ \forall h \in H \mbox{ and almost every $g\in G$}\}, 
\end{equation*}
with distance given by $\|f_1-f_2\|_p^p=\sum_{x\in G/H}\|f_1(x)-f_2(x)\|_p^p$. The induced affine isometric action $\tilde{A}$ of $G$ on $E$ is then given by $(\tilde{A}(g))f(g')=f(g^{-1}g')$, for $f\in E,\,g,g'\in G$.

A function $\xi_0\in E$ is then defined as follows. Let $s_1=e, s_2,,...,s_n$ be a set of representatives for the left cosets of $H$ in $G$. Set $\xi_0(s_ih)=b(h^{-1})$, for $h\in H,\,i=1,...,n$. Define the 1-cocycle $\tilde{b}$ on $G$ by $\tilde{b}(g)=\tilde{A}(g)\xi_0-\xi_0$, for $g\in G$. For an $h\in H$, we then have:
\begin{equation*}
\|\tilde{b}(h)\|_p^p=\sum_{i=1}^n\|\xi_0(h^{-1}s_i)-\xi_0(s_i)\|_p^p=\sum_{i=1}^n\|\xi_0(h^{-1}s_i)\|_p^p\geq \|\xi_0(h^{-1})\|_p^p=\|b(h)\|_p^p. 
\end{equation*}
Set $K=\max_{1\leq i\leq n}\|\tilde{b}(s_i)\|_p$. Take $T=S\cup\{s_1,...,s_n\}$ as a compact generating set of $G$. For $g\in G$, write $g=s_ih$ for $1\leq i\leq n,\,h\in H$. Then
\begin{multline*}
\|\tilde{b}(g)\|_p\geq\|\tilde{b}(h)\|_p-K\geq \|b(h)\|_p-K\geq C|h|_S^\alpha-K\geq C|h|_T^\alpha-K \\
\geq C(|g|_T-1)^\alpha - K\geq C'|g|_T^\alpha - K'.
\end{multline*}
So the compression of the 1-cocycle $\tilde{b}$ is at least $\alpha$, hence $\alpha_p^\#(G)\geq\alpha_p^\#(H)$.
\end{proof}
The following proof uses a construction by S.R. Gal, see page 4 of \cite{Gal2004}.
\begin{thm}
 Let $V_1$ and $V_2$ be closed subspaces of a Hilbert space. Suppose $H$ is a finite index subgroup of $G_1$ and $G_2$ and suppose there are proper affine isometric actions $\beta_i$ (with compression $\alpha_i$) of each $G_i$ on  $V_i$. Assume that $W<V_1\cap V_2$ is invariant under the actions $(\beta_i|_H)$ and moreover that both these (restricted) actions coincide on $W$. Then $\alpha_2^{\#}(G_1 \ast_H G_2) \geq \frac{\min(\alpha_1,\alpha_2)}{2}$. In particular, $\alpha_2^\#(G_1*_HG_2)\geq \frac{\alpha_2^\#(H)}{2}$. \label{thm:amalgamated}
 \end{thm}

\begin{proof}
 
Following \cite{Gal2004}, let us build a Hilbert space $W_{\Gamma}$ on which $\Gamma = G_1 \ast_H G_2$ acts affinely and isometrically. Let $\omega$ be a finite alternating sequence of 1's and 2's and suppose $\pi$ is a linear action of $H$ on some Hilbert space denoted $\mathcal{H}_{\omega}$. One can induce up the linear action from $H$ to $G_i$, obtaining a Hilbert space
\begin{equation*}
V \mathrel{\mathop:}= \set{f \colon G_i \to \mathcal{H}_\omega \ | \ \forall h\in H, \ f(gh) = \pi(h^{-1})f(g) }\\
\end{equation*}
and an orthogonal action $\pi_i \colon G_i \to \mathcal{O}(V)$ defined by $\pi_i(g) f(g') = f(g^{-1}g')$. The subspace 
\begin{equation*}
\set{f \colon G_i \to \mathcal{H}_\omega \ | \ \forall h\in H,  \ f(h) = \pi(h^{-1})f(1), f_{\mid G_i\backslash H}=0 } 
\end{equation*}
can be identified with $\mathcal{H}_{\omega}$ by letting an element $f$ correspond to $f(1)$. It is clear that the action $\pi_i$ restricted to $H$ coincides with the original linear action $\pi$ via this identification.

So, starting from any linear $H$-action on a Hilbert space $\mathcal{H}_{\omega}$, we can obtain a linear action of say $G_1$ on a Hilbert space that can be written as $\mathcal{H}_\omega \oplus \mathcal{H}_{1\omega}$ for some $\mathcal{H}_{1\omega}$. We can restrict this action to a linear $H$-action on $\mathcal{H}_{1\omega}$ and we can lift this to an action of $G_2$ on a space $\mathcal{H}_{1\omega}\oplus \mathcal{H}_{21\omega}$ and so on, repeating the process indefinitely. Here, we will execute this infinite process twice.

The first linear $H$-action on which we apply the process is obtained as follows. As $\beta_i(H)(W)=W$ for each $i=1,2$, the restriction to $H$ of $\beta_1$, gives naturally a linear $H$-action on $\mathcal{H}_1:=V_1/W$. The second linear $H$-action is obtained by similarly noting that the restriction to $H$ of $\beta_2$ gives a linear $H$-action on $\mathcal{H}_2:=V_2/W$. We then apply the above process indefinitely.
\begin{equation*}
 \mathcal{H}_{1}^{\bullet} \mathrel{\mathop:}=  \rlap{$\overbrace{\phantom{\mathcal{H}_1 \oplus \mathcal{H}_{21}}}^{G_2 \curvearrowright}$} \mathcal{H}_1 \oplus \rlap{$\underbrace{\phantom{\mathcal{H}_{21} \oplus \mathcal{H}_{121}}}_{G_1 \curvearrowright}$} \mathcal{H}_{21} \oplus \rlap{$\overbrace{\phantom{\mathcal{H}_{121} \oplus \mathcal{H}_{2121}}}^{G_2 \curvearrowright}$} \mathcal{H}_{121} \oplus \underbrace{ \mathcal{H}_{2121} \oplus \cdots} , \quad  
 \mathcal{H}_{2}^{\bullet} \mathrel{\mathop:}=  \rlap{$\overbrace{\phantom{\mathcal{H}_2 \oplus \mathcal{H}_{12}}}^{G_1 \curvearrowright}$} \mathcal{H}_2 \oplus \rlap{$\underbrace{\phantom{\mathcal{H}_{12} \oplus \mathcal{H}_{212}}}_{G_2 \curvearrowright}$} \mathcal{H}_{12} \oplus \rlap{$\overbrace{\phantom{\mathcal{H}_{212} \oplus \mathcal{H}_{1212}}}^{G_1 \curvearrowright}$} \mathcal{H}_{212} \oplus \underbrace{ \mathcal{H}_{1212} \oplus \cdots},
\end{equation*}
where for $\omega$ a sequence of alternating 1's and 2's, $G_i$ acts on $\mathcal{H}_{\omega} \oplus \mathcal{H}_{i \omega}$. Note that there are two $H$-actions on $\mathcal{H}_{1}^{\bullet}$ as $H$ acts on the first term $\mathcal{H}_1$. One can verify that both $H$-actions coincide (this fact is also mentioned in \cite{Gal2004},page 4). The same is true for $\mathcal{H}_{2}^{\bullet}$. 

Denote $\mathcal{H}_1^{\circ}=\mathcal{H}_{1}^{\bullet} \ominus \mathcal{H}_1$ and similarly, set $\mathcal{H}_2^{\circ}=\mathcal{H}_{2}^{\bullet} \ominus \mathcal{H}_2$. We denote
\[ W_{\Gamma} = W \oplus \mathcal{H}_1^{\bullet} \oplus \mathcal{H}_2^{\bullet}= V_1 \oplus \mathcal{H}_1^{\circ} \oplus \mathcal{H}_{2}^{\bullet}= V_2 \oplus \mathcal{H}_{2}^{\circ} \oplus \mathcal{H}_{1}^{\bullet}.\]
The above formula, which writes $W$ as a direct sum in three distinct ways, shows that both $G_1$ and $G_2$ act on $W_{\Gamma}$. As mentioned before, the actions coincide on $H$ and so we obtain an affine isometric action of $\Gamma$ on $W_{\Gamma}$. Note that the corresponding $1$-cocycle, when restricted to $G_1$ (or $G_2$), coincides with the $1$-cocycle of $\beta_1$(or $\beta_2$).

We inductively define a length function $\psi_T \colon \Gamma \to \mathbb{N}$ by $\psi_T(h) = 0$ for all $h \in H$ and $\psi_T(\gamma) = \min \set{\psi_T(\eta) + 1 \ | \ \gamma = \eta g, \ \mathrm{where } \ g \in G_1 \cup G_2}$. By applying Proposition 2 in \cite{Harpe1989} to the Bass--Serre tree of $G_1*_HG_2$, we see that this map is conditionally negative definite and thus the normed square of a $1$-cocycle associated to an affine isometric action of $\Gamma$ on a Hilbert space.

Let $\psi_{\Gamma}$ be the conditionally negative definite function associated to the action of $\Gamma$ on $W_{\Gamma}$. We now find the compression of the conditionally negative definite map $\psi = \psi_{\Gamma} + \psi_T$. First set
\[ M = \max \set{|t_j^i|_{G_i} : i=1,2 \ \mathrm{and} \ 1 \leq j \leq [G_i : H]},\]
where $t_j^i$ are right coset representatives of $H$ in $G_i$ such that $t^i_1 = 1_{G_i}$ for $i =1,2$.

Denote $\alpha=\min(\alpha_1,\alpha_2)$ and fix some $\varepsilon>0$ arbitrarily small.
 Let $\gamma \in \Gamma$ and suppose in normal form $\gamma = g t^{i_1}_{j_1} \cdots t^{i_k}_{j_k}$, where $g \in G_i$ for some $i = 1,2$. Assume first that $\psi_T(\gamma) \geq \frac{|\gamma|^{\alpha - \varepsilon}}{M}$. In that case, $\psi(\gamma) \geq \frac{|\gamma|^{\alpha - \varepsilon}}{M}$. Else, we have that $\psi_T(\gamma) < \frac{|\gamma|^{\alpha - \varepsilon}}{M}$ and so for all $\gamma \in \Gamma$ such that $|\gamma|$ is sufficiently large, we have
\begin{align*}
 \psi(\gamma) & \geq  \psi_{\Gamma}(\gamma)  = \norm{\gamma \cdot 0}^2\\
  & \geq  (\norm{g \cdot 0} - \psi_T(\gamma)M)^2 \\
 & \gtrsim ((|\gamma| - \psi_T(\gamma)M)^{\alpha - \varepsilon/2} - \psi_T(\gamma)M)^2\\
 & \geq  ((|\gamma| - |\gamma|^{\alpha - \varepsilon})^{\alpha - \varepsilon/2} - |\gamma|^{\alpha- \varepsilon})^2 \\
 & \gtrsim |\gamma|^{2\alpha- \varepsilon},
\end{align*}
where $\gtrsim$ represents inequality up to a multiplicative constant; we use here that one can always assume, without loss of generality, that the $1$-cocycles associated to $\beta_1$ and $\beta_2$ satisfy $\|b_i(g_i)\|\gtrsim |g_i|^{\alpha-\varepsilon}$ (see Lemma 3.4 in \cite{AntolinDreesen}).

So now, by the first case, $\psi(\gamma) \geq |\gamma|^{\alpha- \varepsilon}$ for all $\gamma \in \Gamma$ that are sufficiently large. Hence, we obtain the lower bound $\alpha^{\#}_2(\Gamma) \geq \alpha^{\#}_2(H) / 2$.
\end{proof}

 \section*{Acknowledgement}
The authors would like to thank Jacek Brodzki for interesting discussions and the referee for their helpful recommendations and for suggesting a simplier example in Example \ref{necessity}. The first author thanks Martin Finn-Sell for interesting conversations on affine subspaces. The second author thanks Micha\l{}  Marcinkowski for interesting discussions related to Gal's paper \cite{Gal2004}.
\bibliographystyle{plain}
\bibliography{/home/chris/Documents/git_repos/texfiles/papers/Bibliography2.bib}
\end{document}